\documentclass[12pt]{article}
\usepackage{amssymb,amsmath,amscd,epsfig,amsthm,dsfont}
\usepackage{dsfont}
\usepackage{hyperref}
\usepackage{fullpage,verbatim} 

\usepackage{lineno,hyperref}
\modulolinenumbers[5]

\newcommand{\Aut}{\mathop{\rm Aut}\nolimits}
\newcommand{\St}{\mathop{\rm Stab}\nolimits}

\newcommand{\A}{\mathcal A}

\newcommand{\X}{\mathcal X}
\newcommand{\Y}{\mathcal Y}
\newcommand{\LL}{\mathcal L}
\newcommand{\Z}{\mathbb Z}
\newcommand{\R}{\mathbb R}
\newcommand{\Sym}{\mathop{\rm Sym}\nolimits}

\newcommand{\Orb}{\mathop{\rm Orb}\nolimits}

\newtheorem{theorem}{Theorem}[section]
\newtheorem{corollary}[theorem]{Corollary}
\newtheorem{remark}[theorem]{Remark}
\newtheorem{proposition}[theorem]{Proposition}
\newtheorem{lemma}[theorem]{Lemma}
\newtheorem{definition}{Definition}
\newtheorem{question}{Question}
\newtheorem{example}{Example}

\usepackage{xcolor}

\begin{document}
\title{Ergodic decomposition of group actions on rooted trees}

\author{Rostislav Grigorchuk\footnote{The first author was partially supported by the NSF grant DMS-1207699}\\
               Department of Mathematics,\\
               Texas A\&M University,\\
               College Station, TX, 77843\\
               \href{mailto:grigorch@math.tamu.edu}{grigorch@math.tamu.edu}\\
               \and
        Dmytro Savchuk\footnote{The second author was partially supported by the New Researcher Grant and the Proposal Enhancement Grant from USF Internal Awards Program.}\\
               Department of Mathematics and Statistics\\
               University of South Florida\\
               4202 E Fowler Ave\\
               Tampa, FL 33620-5700\\
               \href{mailto:savchuk@usf.edu}{savchuk@usf.edu}
}


\maketitle

\abstract{We prove a general result about the decomposition on ergodic
components of group actions on boundaries of spherically homogeneous
rooted  trees. Namely, we identify the space of ergodic components with
the boundary of the orbit tree associated with the action, and show that
the canonical system of ergodic invariant probability measures coincides
with the system of uniform measures on the boundaries of minimal
invariant  subtrees of the tree.

A special attention is given to the case of groups generated by  finite
automata. Few examples, including the lamplighter group,  Sushchansky
group, and the, so  called, Universal group are considered in order to
demonstrate applications of the theorem.}

\section*{Introduction}

The ergodic decomposition theorem is one of the most important and frequently used theorems in dynamical systems and ergodic theory. It was initiated by von Neumann, Bogolyubov and Krylov but, perhaps, its first precise form was given by Rokhlin~\cite{rokhlin:ergodic_decomp49}, where he introduced the class of measure spaces now called the Lebesgue spaces.

At first, the ergodic theorem was proved for the case of one automorphism of a Lebesgue space or a one parameter family of such automorphisms, which corresponds to the actions of groups $\mathbb Z$ or $\mathbb R$ respectively. Later, the theorem was extended to the case of countable groups and locally compact groups (and further generalizations were made including passing from finite to infinite measures, from invariant to quasi-invariant measures, and from locally compact groups to some classes of non-locally compact groups~\cite{bufetov:ergodic_decomposition14}).

In 1961 Varadarajan~\cite{varadarajan:groups_of_automorphisms63} (see also Farrell~\cite{farrell:representations_invar_measures62}) proved the ergodic decomposition theorem in the topological setting, namely when a group $G$ acts on a Polish space by homeomorphisms. Varadarajan's theorem (Theorem~\ref{thm:erg_decomp_general} in Section~\ref{sec:prelim}) describes ergodic decomposition for each $G$-invariant probability measure.

The main goal of this article is to show how Varadarajan's theorem works in the situation when a group $G$ acts by automorphisms on a spherically homogeneous rooted tree $T$ and, consequently, by homeomorphisms on its boundary $\partial T$ (which is homeomorphic to the Cantor set as soon as the tree has infinitely many ends). For any such action of $G$ each level of the tree $T$ is an invariant subset. The uniform probability measure $\mu_T$ on $\partial T$ is invariant with respect to the whole group $\Aut(T)$ of automorphisms of the tree and the ergodicity of the system $(G,\partial T,\mu_T)$ is equivalent to level transitivity of the action $(G,T)$~\cite[Proposition~6.5]{gns00:automata} (and also is equivalent to unique ergodicity).
This situation has also a direct connection to the theory of profinite groups. Namely, if
a group $G$ acts transitively on the levels of a tree, then its closure $\overline{G}$ in $\Aut(T)$, which is a profinite group, acts transitively on the boundary $\partial T$, and the uniform measure $\mu_T$ becomes the image of the Haar measure on $G$. In this case, the dynamical system $(G,\partial T,\mu_T)$ is isomorphic to the system $(G, \overline{G}/P,\mu_T)$, where $P =\mathrm{stab}_{\overline{G}}(\xi)$ is the stabilizer of point $\xi\in\partial T$ under the action of $\overline{G}$. The converse is also true in the following sense. By the result of Mackey~\cite{mackey:ergodic64}, any action $(G,X,\mu)$ with pure point spectrum, where $G$ is a countable group acting faithfully on $X$ by transformations preserving the probability measure $\mu$, is isomorphic to the action of type $(G',K/P,\lambda)$, where $K$ is a profinite group, $G'$ is a subgroup of $K$ isomorphic to $G$, and $P$ is a closed subgroup of $K$.
In turn, as shown in Theorem 2.9 in~\cite{grigorch:dynamics11eng}, the latter action is isomorphic to the action $(G,\partial T,\nu)$, where a spherically homogeneous rooted tree $T$ is constructed as the coset tree of a family of open subgroups of $K$ whose intersection is $P$, and $\nu$ is the uniform measure on $\partial T$. Therefore, the profinite case in Mackey's theorem corresponds precisely to the action on rooted trees.

In the case when the action of $G$ on $T$ is not level transitive the situation is more complicated. In order to decompose $\mu_T$ into ergodic components and describe all $G$-invariant ergodic probability measures on $\partial T$ one needs to know the structure of the \emph{orbit tree} $T_G$ whose vertices are orbits of $G$ on the set of vertices $V(T)$ of $T$ and the adjacency relation is induced by the adjacency in $T$ (i.e., $T_G$ is simply the quotient graph of $T$ under the action of $G$). This tree was used in~\cite{gawron_ns:conjugation} in order to give a criterion for establishing when two elements are conjugate in $\Aut(T)$, as well as recently in~\cite{klimann:finiteness} to deal with the finiteness problem in automaton groups generated by invertible-reversible automata. In Theorem~\ref{thm:EI_homeo_bndry} we show that the boundary $\partial T_G$ of the orbit tree can be naturally identified with the space of ergodic components of the action of $G$ on $\partial T$: there is a bijection between $\partial T_G$ and the set of the minimal invariant subtrees of $T$, and uniform probability measures on boundaries of these trees are exactly all ergodic invariant probability measures for the system $(G,\partial T)$. In Section~\ref{sec:examples} we apply the obtained results to get the ergodic decompositions for actions of some groups generated by finite automata.

The class of automaton groups possesses a number of interesting and unusual algebraic and dynamical properties. There are many examples showing that even simple automata (with a small number of states and an alphabet consisting of just two symbols) demonstrate very complicated algebraic, combinatorial, and dynamic behavior~\cite{grigorch:burnside,gns00:automata,grigorch_z:basilica,bondarenko_gkmnss:full_clas32_short,grigorch_s:hanoi_spectrum}.

After considering simple examples of ergodic decompositions of actions of finite groups and level transitive actions, we deal with Sushchansky infinite $p$-groups~\cite{sushch:burnside} in Subsection~\ref{ssec:sushch} and the universal group~\cite{grigorch:solved} for the family of groups $G_{\omega}$ from~\cite{grigorch:degrees} in Subsection~\ref{ssec:univ_grigorch}.

The most complicated example is studied in Subsection~\ref{ssec:lamp} and deals with the 2-state automaton over 2-letter alphabet generating the lamplighter group $\LL$. The automaton presentation of $\LL$ was found in~\cite{gns00:automata} and was used in~\cite{grigorch_z:lamplighter} to compute the spectrum of the discrete Laplacian, which happened to be purely discrete. This automaton presentation of $\LL$ is given on a binary tree, which by Lemma~3 in~\cite{bondarenko_gkmnss:full_clas32_short} implies that its action on this tree is spherically transitive. Therefore by Proposition 6.5 in~\cite{gns00:automata} there is only one (ergodic) invariant probability measure on $\partial T$. However, it is more interesting in this case to consider the ergodic decomposition of actions of subgroups of $\LL$ that do not act level transitively. In particular, we give a complete description of such decompositions for cyclic subgroups $\langle a\rangle$ and $\langle b\rangle$, where $a$ and $b$ are the automorphisms of the tree corresponding to the states of the generating automaton. In order to get the structure of the orbit trees in these cases, we explicitly describe how each orbit looks like using the representation of the lamplighter group by functions that act on formal power series.

We hope that the considerations initiated in this article will be useful for further investigations of group actions on trees and for solving the classification problems started in~\cite{gns00:automata,bondarenko_gkmnss:full_clas32_short}.

The structure of the paper is the following. In Section~\ref{sec:prelim} we recall basic definitions and set up the notation. The main theorem is proved in Section~\ref{sec:theorem}. We conclude the paper with several particular examples in Section~\ref{sec:examples}.

\noindent \textbf{Acknowledgement.} The authors would like to thank the referee for the very detailed careful review with numerous suggestions that greatly improved the exposition of the paper.

\section{Preliminaries}
\label{sec:prelim}

In this paper we will deal only with rooted trees, i.e. the trees with a distinguished vertex called the \emph{root}. For each such tree $T$ and $n\geq 0$ the set $[T]_n$ of vertices of $T$ at combinatorial distance $n$ from the root is called the $n$-th \emph{level} of $T$. For each vertex $v$ of $T$ of the $n$-th level the vertices of the $(n+1)$-st level adjacent to $v$ are called the \emph{children} of $v$. We will visualize the rooted trees as growing down with the root on top. In this visualization the children of a vertex are the vertices that are right below it.

For each rooted tree $T$ the boundary $\partial T$ of $T$ is defined as the set of all infinite paths in $T$ starting from the root that do not have backtracking. A tree is called \emph{spherically homogeneous} if the degrees of all vertices of each level coincide (but this common degree may depend on the level). A special very important class of spherically homogeneous rooted trees is the class of regular rooted trees. A rooted tree is called \emph{regular} if each vertex of the tree has the same number of children. If each vertex has $d$ children, the tree is called $d$-regular rooted tree (or $d$-ary rooted tree) and is denoted by $T_d$. The tree $T_2$ is called \emph{binary} and is depicted in Figure~\ref{fig:binary_tree}.

\begin{figure}[h]
\begin{center}
\epsfig{file=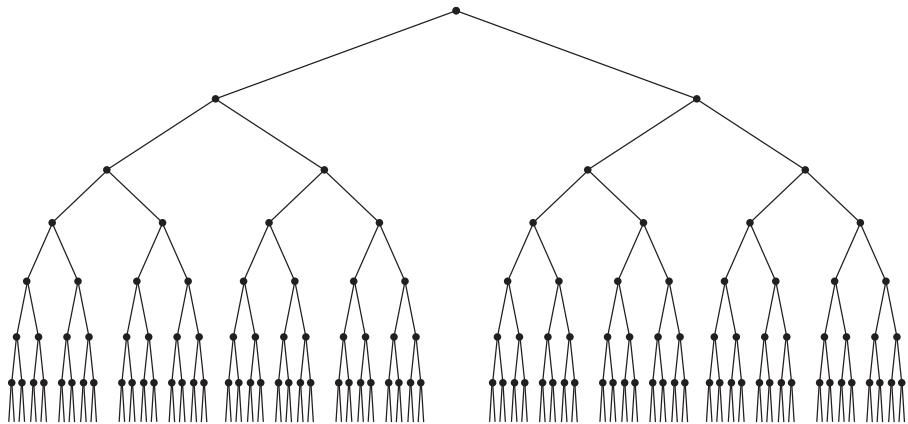,width=340pt}
\end{center}
\caption{Binary tree $T_2$\label{fig:binary_tree}}
\end{figure}

The class of regular rooted trees naturally arises in symbolic dynamics. Indeed, let $\Sigma$ be a finite alphabet of cardinality $d$. We will denote by $\Sigma^*$ and $\Sigma^{\omega}$ the sets of all finite and infinite words over $\Sigma$, respectively. The set $\Sigma^*$ can be naturally identified with the set of vertices of the $d$-ary rooted tree $T_d$, where the empty word $\emptyset\in\Sigma^*$ corresponds to the root of the tree and words $v$ and $vx$ for $v\in\Sigma^*$ and $x\in\Sigma$ are declared to be adjacent. With this identification between $\Sigma^*$ and $T_d$ the set $\Sigma^\omega$ is naturally identified with the boundary $\partial T_d$. The set $\Sigma^n$ of words over $\Sigma$ of length $n$ constitutes the $n$-th level of $\Sigma^*$. For a word $v\in\Sigma^*\cup\Sigma^\omega$ we will denote by $|v|\in\mathbb N\cup\{0,\infty\}$ the length of $v$.

For a rooted tree $T$ and point $\xi\in\partial T$ we denote by $[\xi]_n$ the vertex of $T$ located on a path $\xi$ at the distance $n$ from the root. In the particular case of the regular rooted tree $T=\Sigma^*$ and $\xi=x_1x_2x_3\ldots$ for $x_i\in\Sigma$ we have $[\xi]_n=x_1x_2\ldots x_n$.

Let $G$ be a group acting on a rooted tree $T$ by automorphisms preserving the root. Then, this action preserves the levels of the tree. We say that $(G,T)$ is \emph{spherically transitive} if it is transitive on each level of $T$. A necessary condition for the action to be spherically transitive is that the tree $T$ has to be spherically homogeneous.

In our study the central role is played by the following notion.

\begin{definition}
The \emph{orbit tree} $T_{G}$ for the action of $G$ on a rooted tree $T$ (i.e. on the set of vertices $V(T)$ of $T$) is the graph whose vertices correspond to the orbits of $G$ on the levels of $T$, in which two orbits are adjacent if and only if they contain vertices that are adjacent in $T$.
\end{definition}

It follows directly from the definition that the orbit tree $T_G$ is again a rooted tree with the root corresponding to the 1-element orbit consisting of the root of $T$. Indeed, suppose vertices $v$ and $w$ belong to the $n$-th level $[T]_n$ of $T$ and let $v'$ and $w'$ be vertices of $[T]_{n-1}$ adjacent to $v$ and $w$, respectively. If $v$ and $w$ belong to the same orbit of $G$, then there is $g\in G$ that moves $v$ to $w$. In this case the same $g$ necessarily moves $v'$ to $w'$. Thus, each vertex of the $n$-th level of $T_G$ is adjacent to exactly one vertex of the previous level. However, $T_G$ may be not spherically homogeneous even if $T$ is spherically homogeneous. Orbit trees in various forms have been studied earlier (see, for example,~\cite{gawron_ns:conjugation,bondarenko_s:sushch,klimann:finiteness,klimann_ps:3state}). They describe the partition of the set of vertices of a rooted tree into transitive components under the action of a group.

There is a natural map $\psi\colon V(T)\to V(T_{G})$ that sends a vertex of $T$ to its orbit viewed as a vertex of $T_{G}$. This map naturally extends to a continuous map $\psi\colon \partial T\to\partial T_{G}$ with respect to the topologies that we define below.

The boundary $\partial T$ of a rooted tree may be viewed as an (ultra)metric space as follows: fix a monotonically decreasing sequence $\{\lambda_n\}_{n\geq0}$ converging to 0 and define the distance of two points in $\partial T$ to be equal to $\lambda_k$,  where $k$ denotes the length of the longest common part of the two (geodesic) paths connecting the root to each of them.
This metric defines a topology on $\partial T$ that in the case of a spherically homogeneous rooted tree can be constructed in the following way. The set of vertices of a spherically homogeneous rooted tree $T_{\{\Sigma_n, n\geq0\}}$ can be identified with
\[\bigcup_{n\geq 0}\prod_{i=0}^n\Sigma_i,\]
where $\Sigma_n$ is an alphabet of cardinality that is equal to the number of children of each vertex of level $n$. The boundary of $\partial T_{\{\Sigma_n, n\geq0\}}$ of this tree is naturally identified with $\prod_{n\geq 0}\Sigma_n$ that is endowed with a Tychonoff product topology (when using the discrete topologies on $\Sigma_n$, $n\geq 0$). The topological structure induces the Borel structure on $\partial T$. In the case of spherically homogeneous tree $T$ one can construct the uniform probability measure $\mu_T$ on $\partial T$ by defining
\[\mu_T(C_v)=\frac1{|[T]_{|v|}|},\]
where for a vertex $v\in V(T)$ the cylindrical set $C_v$ consists of all infinite paths in $\partial T$ that go through $v$. This is the measure whose existence and uniqueness is proved in the Kolmogorov consistency (also called extension, or existence) theorem~\cite{kolmogoroff:grundbegriffe,parthasarathy:prob_measures67}. In the case of a regular tree the uniform probability measure on its boundary coincides with the Bernoulli measure.

\begin{lemma}
The map $\psi\colon \partial T\to\partial T_{G}$ is a continuous surjective map.
\end{lemma}

\begin{proof}
A basis of the topology in $\partial T_{G}$ consists of cylindrical sets $C_{\mathcal O}, \mathcal O\in V(T_{G})$ consisting of all infinite paths in $T_G$ that go through a vertex $\mathcal O$ of $T_{G}$ (i.e., $\mathcal O$ represents an orbit of $G$ on some level of $T$). Therefore,
\[\psi^{-1}(C_{\mathcal O})=\bigcup_{v\in\mathcal O}C_v\]
is open in $\partial T$ (in fact, it is clopen) and, hence, $\psi$ is continuous.
\end{proof}

Similarly to the boundary of a rooted tree, the whole group $\Aut(T)$ of all automorphisms of a rooted tree $T$ can be naturally endowed with a topology, induced by the metric $\lambda(\alpha,\gamma)=\lambda_k$, where $\{\lambda_n\}$ is again any monotonically decreasing sequence converging to 0, and $k$ is the largest number of the level of $T$ on which the actions of the automorphisms $\alpha$ and $\gamma$ coincide. Note, that the topology defined by such metric does not depend on the choice of $\{\lambda_n\}$.


By a \emph{measure} on a standard Borel space $X$ we will mean a non-zero Borel measure on $X$. A measure $\mu$ on $X$ is called \emph{probability measure} if $\mu(X)=1$. With the above described topology $\Aut(T)$ is a compact totally disconnected group (hence, a \emph{profinite group}, i.e. a group isomorphic to the inverse limit of an inverse system of discrete finite groups) acting on $\partial T$ and, in the case when $T$ is spherically homogeneous, preserving the uniform probability measure $\mu_T$. Moreover, the converse is true in the following sense.

\begin{proposition}[\mbox{see~\cite[Proposition 2]{grigorch:jibranch}}]
Every countably based profinite group is isomorphic to a closed subgroup of $\Aut(T)$ for some spherically homogeneous rooted tree $T$.
\end{proposition}

\begin{proof}
Let $G$ be a countably based profinite group. By definition it has a countable descending sequence $G=V_0>V_1>V_2>\cdots$ of finite index open subgroups with trivial intersection. Then $G$ acts faithfully by automorphisms on the, so-called, \emph{coset tree} $T$ of the sequence $\{V_n\}_{n\geq 0}$ constructed as follows. The vertices of $T$ correspond to the cosets of $V_n$ in $G$ for all $n\geq 0$. Two vertices corresponding to cosets $V_{n}g$ and $V_{n+1}h$ are adjacent if and only if $V_{n}g\supset V_{n+1}h$. Then $G$ acts on $T$ by automorphisms simply by right multiplication: an element $g\in G$ sends $V_nh$ to $V_nhg$. This action is clearly faithful since the kernel is equal to the trivial $\cap_{n\geq 0}V_n$.
\end{proof}


Let $G$ be a locally compact group acting on a standard Borel space $X$ by transformations preserving a probability measure $\mu$. Measure $\mu$ is called \emph{ergodic} if the measure of each $G$-invariant Borel set in $X$ is either 0 or 1.
We denote by $\mathcal M_G$ the space of all invariant probability measures on $X$ and by $\mathcal{M}^e_G$ the set of all ergodic invariant probability measures on $X$. Both $\mathcal M_G$ and $\mathcal{M}^e_G$ are Borel subsets of the standard Borel space $P(X)$ of all probability measures on $X$. Recall that $P(X)$ is endowed with the \emph{weak topology} (sometimes called \textit{weak$^*$ topology}): a sequence of measures $\mu_n\in P(X)$ \emph{weakly converges} to a measure $\mu\in P(X)$ if for each bounded continuous function $f\colon X\to\mathbb R$ we have
\[\int f\,d\mu_n\to\int f\,d\mu,\quad n\to\infty.\]
In the case when $G$ is a countable discrete group, an invariant measure in $\mathcal M_G$ is ergodic if and only if it is an extreme point in the (Choquet) simplex $\mathcal{M}_G$, i.e. it cannot be written as a convex combination of other invariant measures from $\mathcal{M}_G$ with non-zero coefficients. However, this is not true for general locally compact groups as Kolmogorov's example shows~\cite{grigorch_h:amenability_ergodic_top_groups,bufetov:ergodic_decomposition14,fomin:measures1950}.

The ergodic decomposition theorem due to Varadarajan~\cite{varadarajan:groups_of_automorphisms63} stated as in Kechris Miller\cite{kechris_m:topics_in_orbit_equivalence04} (see also~\cite{farrell:representations_invar_measures62}) states:

\begin{theorem}
\label{thm:erg_decomp_general}
For a locally compact second countable group $G$ let $X$ be a standard Borel $G$-space and let $\mathcal M_G$ and $\mathcal{M}^e_G$ be the spaces of all invariant probability measures on $X$ and ergodic invariant probability measures on $X$, respectively. Suppose $\mathcal M_G\neq\emptyset$. Then $\mathcal{M}^e_G\neq\emptyset$ and there is a Borel surjection $\pi\colon X\to \mathcal{M}^e_G$ such that
\begin{itemize}
\item[1)] $\pi$ is $G$-invariant (i.e., $\pi$ is constant on each orbit of $G$),
\item[2)] For $\nu\in\mathcal{M}^e_G$ the set $X_{\nu}=\{x\in X\colon \pi(x)=\nu\}$ satisfies $\nu(X_\nu)=1$ and the action $G\curvearrowright X_{\nu}$ has a unique invariant measure, namely $\nu$, and
\item[3)] if $\mu\in\mathcal M_G$, then $\mu=\int\pi(\xi)\,d\mu(\xi)$.
\end{itemize}
Moreover, $\pi$ is uniquely determined in the sense that, if $\pi'$ is another such map, then the set $\{x\in X\colon\pi(x)\neq\pi'(x)\}$ has measure zero with respect to all measures in $\mathcal M_G$.
\end{theorem}

Throughout the paper we will use the above theorem in two cases: when a group $G$ is countable with the discrete topology, and when $G$ is a profinite group.

\section{Ergodic decomposition for groups acting on rooted trees}
\label{sec:theorem}

Let $G$ be a group acting on a rooted tree $T$ by automorphisms and, hence, on its boundary $\partial T$ by homeomorphisms. Throughout this section we will write $\X=\partial T$ and $\Y=\partial T_{G}$ where $T_{G}$ is the orbit tree associated with the action of $G$ on $T$.

\begin{definition}
A \emph{leaf} of a rooted tree is a vertex of degree one which is different from the root of the tree.
\end{definition}

All rooted trees that we consider in this paper are rooted trees with no leaves (i.e., each vertex lies on some path(s) in the boundary of the tree).

\begin{definition}
Let $T$ be an infinite rooted tree with no leaves. A subtree of $T$ with no leaves is called \emph{rooted} if it contains the root of $T$.
\end{definition}

\begin{definition}
Let $G$ be a group acting on a rooted tree $T$ with no leaves. A nonempty rooted subtree $T'$ of $T$ with no leaves is called \emph{minimal} (denoted $T'\prec T$) if it is a minimal (with respect to inclusion) invariant subtree with no leaves.
\end{definition}

\begin{proposition}
For a group $G$ acting on a rooted tree $T$, the boundary $\partial T$ can be decomposed as
\begin{equation}
\label{eq:boundary}
\partial T=\bigsqcup_{T'\prec T}\partial T'.
\end{equation}
Moreover, there is a bijection between the set of minimal subtrees of $T$ and the boundary $\partial T_{G}$ of the orbit tree $T_{G}$ associated with the action of $G$ on $T$.
\end{proposition}

\begin{proof}
First we show that if $T'$ and $T''$ are two different minimal subtrees, then $\partial T'\cap\partial T''=\emptyset$. Indeed, since $T'$ and $T''$ are minimal subtrees of $T$, on each level of $T$ their sets of vertices must either coincide or be disjoint: if $\xi\in\partial T'\cap\partial T''$, then $[\xi]_n$ is a common vertex of the $n$-th levels of $T'$ and $T''$ and therefore, since there is such a vertex for each $n$, minimality ensures that $T'=T''$.

Now, for each $\xi\in\partial T$ we will build a minimal subtree $T_{\xi}$ of $T$ with $\xi\in\partial T_{\xi}$. Define $V(T_{\xi})$ to be the preimage under $\psi$ of the set of vertices $\{[\psi(\xi)]_n\colon n\geq0\}$ of the orbit tree $T_{G}$. In other words, $T_{\xi}$ is a union of orbits of $[\xi]_n$ under the action of $G$. Then by construction $T_{\xi}$ is a minimal subtree of $T$ containing $\xi$. Moreover, if $T'$ is a minimal subtree of $T$, its boundary must contain some point $\eta\in\partial T$ that is also contained in $T_{\eta}$, yielding $T'=T_{\eta}$. Finally, the fact that $T_{\xi}=T_{\xi'}$ if and only if $\psi(\xi)=\psi(\xi')$ proves that the map from the boundary $\partial T_{G}$ to the set of minimal subtrees of $T$ sending $\eta\in\partial T_{G}$ to $T_{\eta}=\psi^{-1}\left(\{[\eta]_n\colon n\geq0\}\right)$ is a bijection.
\end{proof}

The last proof motivates the following notation: for $\xi\in\partial T$, and $\eta\in\partial T_{G}$ we associate minimal subtrees $T_{\xi}$ and $T_{\eta}$ of $T$ with $\xi\in\partial T_{\xi}$, $\psi(\partial T_{\eta})=\eta$ and $T_{\xi}=T_{\psi(\xi)}$.

Observe, that the decomposition~\eqref{eq:boundary} can now be rewritten as
\begin{equation}
\label{eq:boundary}
\partial T=\bigsqcup_{\eta\in\partial T_G}\partial T_\eta
\end{equation}
and for each $\eta\in\partial T_G$ the boundary $\partial T_{\eta}$ is a closed subset of $\partial T$.
\begin{theorem}
\label{thm:EI_homeo_bndry}
Let $G$ be a countable discrete or profinite group acting faithfully by automorphisms on a spherically homogeneous rooted tree $T$ and by homeomorphisms on its boundary $\X=\partial T$, and let $\Y=\partial T_G$ be the boundary of the corresponding orbit tree $T_G$.
\begin{itemize}
\item[(a)] The map $\beta\colon \Y\to\mathcal{M}^e_G$ sending a point $\eta\in\Y$ to the uniform probability measure $\mu_{\eta}$ with support $\partial T_{\eta}$ in the space $\mathcal{M}^e_G$ of invariant ergodic measures on $\X$ is a homeomorphism.
\item[(b)] The map $\pi=\beta\circ\psi\colon \X\to\mathcal{M}^e_G$ sending a point $\xi\in\X$ to the uniform probability measure $\mu_{\psi(\xi)}$ with support $\partial T_{\xi}$ satisfies conditions 1)-3) of Theorem~\ref{thm:erg_decomp_general} and thus defines the ergodic decomposition of the action of $G$ on $\X$.
\end{itemize}
\end{theorem}

\begin{proof}
We will first prove the theorem for the case when $G$ is a countable discrete group.

We start from the proof of part (a). The action of $G$ on $T_\eta$ is level transitive since each level $n$ of $T_\eta$ corresponds to exactly one orbit of $G$ on the same level $n$ of $T$. Hence, by~\cite[Proposition~6.5]{gns00:automata}, the action of $G$ on $\partial T_\eta$ (and thus on $\partial T$) is ergodic with respect to $\mu_\eta$, so $\mu_\eta\in\mathcal{M}^e_G$ and $\beta$ is well-defined.

To show that $\beta$ is surjective, assume $\mu\in\mathcal{M}^e_G$. Let $[\mu]_n$ be the measure induced by $\mu$ on the $n$-th level of $T$, i.e., for $A\subset [T]_n$
\[[\mu]_n(A)=\sum_{v\in A}\mu(C_v).\]
In other words, $[\mu]_n$ is the projection of $\mu$ induced by the natural projection $p_n\colon\partial T\to [T]_n$. As $\mu$ is ergodic invariant probability measure, so is $[\mu]_n$. Therefore, $[\mu]_n$ is supported on exactly one orbit $\mathcal O_{n,\mu}$ of $G$ on $[T]_n$ an its value on each vertex of $\mathcal O_{n,\mu}$ is equal to $|\mathcal O_{n,\mu}|^{-1}$. The sequence of orbits $\{\mathcal O_{n,\mu}\}_{n\geq0}$ defines a unique point $\eta\in\Y$ since $\mathcal O_{n,\mu}$ is always adjacent to $\mathcal O_{n+1,\mu}$ in $T_{G}$. By construction we get that $\mu$ coincides with $\mu_\eta$ on each cylindrical set in $\partial T$. Thus, $\mu=\mu_\eta$ and the map $\beta$ is onto.

Finally, we will prove that the topological structure on $\Y$ is isomorphic to the one on $\mathcal{M}^e_G$. Since both $\Y$ and $\mathcal{M}^e_G$ are metrizable, their topologies are completely determined by the convergent sequences.
Thus, it is enough to show that the following conditions are equivalent: (i) $\eta_n\to\eta,\ n\to\infty$ in $\Y$ and (ii) $\mu_{\eta_n}\to\mu_{\eta}$ weakly as $n\to\infty$, i.e. formula
\begin{equation}
\label{eqn:conv}
\int \mathds 1_{C_v}\,d\mu_{\eta_n}\to\int \mathds 1_{C_v}\,d\mu_{\eta},\quad n\to\infty.
\end{equation}
holds for all $v\in V(T)$, where $\mathds 1_{C_v}$ denotes the characteristic function of a cylindrical set $C_v$.

Suppose that $v$ is on the $l$-th level of $T$ and $\eta_n\to\eta,\ n\to\infty$. Then there is $N>0$ such that for all $n\geq N$ we have $[\eta_n]_l=[\eta]_l$. But in this case for $n\geq N$
\begin{multline*}\int \mathds 1_{C_v}\,d\mu_{\eta_n}=
\left\{\begin{array}{ll}
0,&\text{if $v\notin[\eta_n]_l$}\\
|\psi^{-1}([\eta_n]_l)|^{-1},&\text{if $v\in[\eta_n]_l$}
\end{array}\right.\\
=
\left\{\begin{array}{ll}
0,&\text{if $v\notin\psi^{-1}([\eta]_l)$}\\
|\psi^{-1}([\eta]_l)|^{-1},&\text{if $v\in\psi^{-1}([\eta]_l)$}
\end{array}\right.=\int \mathds 1_{C_v}\,d\mu_{\eta}.
\end{multline*}
Therefore convergence~\eqref{eqn:conv} takes place.

Conversely, assume that $\mu_{\eta_n}\to\mu_{\eta},\ n\to\infty$ and suppose that $\eta_n\not\to\eta,\ n\to\infty$. Then there is some level $l$ and a subsequence $\{\eta_{n_i}\}_{i\geq 1}$ such that $[\eta_{n_i}]_l\neq[\eta]_l$. Consider the set
\[A=\bigsqcup_{v\in\psi^{-1}([\eta]_l)}C_v\]
corresponding to all points in $\X$ that pass through vertices in $\psi^{-1}([\eta]_l)$. The characteristic function $\mathds 1_{A}\colon\X\to\R$ satisfies
\[\int \mathds 1_{A}\,d\mu_{\eta_n}=0\neq 1=\int \mathds 1_{A}\,d\mu_{\eta}\]
contradicting our assumptions. This finishes the proof of part (a).

To prove part (b) of the theorem, we first note that $\pi$ is a Borel surjection since it is a composition of a continuous projection $\psi$ and a homeomorphism $\beta$. We will now check conditions 1)-3) one by one. Condition 1) is trivially satisfied by definition of $\pi$.

Condition 2) is satisfied by Proposition 6.5 in~\cite{gns00:automata}, as for each $e\in\mathcal{M}^e_G$ the set $X_e=\pi^{-1}(e)$ simply coincides with $\partial T_\eta$ for $\eta\in\Y$ such that $e=\mu_\eta$.

Finally, condition 3) is proved as follows. Consider $\mu\in \mathcal M_G$ and $A=C_v$, $v\in V(T)$. Let $\mathcal O_v$ denote the orbit of $v$ under $G$. For each $w\in\mathcal O_v$ \begin{equation}
\label{eqn:invar}
\mu(C_w)=\mu(C_v)=\mu(A).
\end{equation}
Also, for $\xi\in\X$ we have
\[
\mu_{\psi(\xi)}(A)=\left\{\begin{array}{ll}
0,&\text{if $\xi$ does not pass through a vertex in $\mathcal O_v$},\\
\frac1{|\mathcal O_v|},&\text{otherwise}.
\end{array}
\right.
\]
Therefore, the right-hand side of the equality in condition 3) applied to the set $A$ can be decomposed as:
\begin{multline*}
\int_{\X}\bigl(\pi(\xi)\bigr)(A)\,d\mu(\xi)=\sum_{w\in\mathcal O_v}\int_{T_w}\mu_{\psi(\xi)}(A)\,d\mu(\xi)=\sum_{w\in\mathcal O_v}\int_{T_w}\frac1{|\mathcal O_v|}\,d\mu(\xi)\\
=\frac1{|\mathcal O_v|}\sum_{w\in\mathcal O_v}\int_{T_w}\,d\mu(\xi)=\frac1{|\mathcal O_v|}\sum_{w\in\mathcal O_v}\mu(T_w)=\frac1{|\mathcal O_v|}\sum_{w\in\mathcal O_v}\mu(A)=\mu(A),
\end{multline*}
where we applied equality~\eqref{eqn:invar} in the next to the last transition.

Finally, we note that the profinite case follows from the previous case because for a countably based profinite group $G$ one can always find a countable discrete dense subgroup in $G$. Note, that $G$ must be countably based since it acts faithfully on $T$ and, hence it is a subgroup of the countably based group $\Aut(T)$.
\end{proof}

\section{Examples of Ergodic Decompositions}
\label{sec:examples}

\subsection{Groups generated by automata}

Most of the interesting examples of groups acting on rooted trees come from the class of groups generated by automata (not to be confused with automatic groups in the sense of~\cite{epstein_chpt:word_processing_in_groups92}). We start by recalling some basic definitions that we shall need later.

\begin{definition}
A \emph{Mealy automaton} (or simply \emph{automaton}) is a tuple
$(Q,\Sigma,\pi,\lambda)$, where $Q$ is a set (the set of states), $\Sigma$ is a
finite alphabet, $\pi\colon Q\times \Sigma\to Q$ is the transition function
and $\lambda\colon Q\times \Sigma\to \Sigma$ is the output function. If the set
of states $Q$ is finite the automaton is called \emph{finite}. If
for every state $q\in Q$ the output function $\lambda(q,x)$ induces
a permutation of $\Sigma$, the automaton $\A$ is called invertible.
Selecting a state $q\in Q$ produces an \emph{initial automaton}
$\A_q$.
\end{definition}

Automata are often represented by their associated \emph{Moore diagrams}. The
Moore diagram of an automaton $\A=(Q,\Sigma,\pi,\lambda)$ is the directed
labelled graph in which the vertices are the states from $Q$ and the labelled edges
have the form $q\stackrel{x|\lambda(q,x)}{\longrightarrow}\pi(q,x)$ for
$q\in Q$ and $x\in \Sigma$. If the automaton is invertible, then it is
common to label vertices of the Moore diagram by the permutation
$\lambda(q,\cdot)\in\Sym(\Sigma)$ and leave just first components from the labels
of the edges, see for example Figure~\ref{fig:lamplighter_aut}. An example of Moore diagram (for Sushchansky automaton) is shown in
Figure~\ref{fig:lamplighter_aut}.

Any initial automaton induces an endomorphism of the rooted tree $\Sigma^*$ (in this situation we consider $\Sigma^*$ specifically as a tree and not as free monoid). Given a word
$v=x_1x_2x_3\ldots x_n\in \Sigma^*$ it scans its first letter $x_1$ and
outputs $\lambda(x_1)$. The rest of the word is handled in a similar
fashion by the initial automaton $\A_{\pi(x_1)}$. Formally speaking,
the functions $\pi$ and $\lambda$ can be extended to $\pi\colon
Q\times \Sigma^*\to Q$ and $\lambda\colon  Q\times \Sigma^*\to \Sigma^*$ via
\[\begin{array}{l}
\pi(q,x_1x_2\ldots x_n)=\pi(\pi(q,x_1),x_2x_3\ldots x_n),\\
\lambda(q,x_1x_2\ldots x_n)=\lambda(q,x_1)\lambda(\pi(q,x_1),x_2x_3\ldots x_n).\\
\end{array}
\]

By construction any initial automaton acts on the rooted tree $\Sigma^*$ as a
endomorphism. In case of invertible automaton it acts as an
automorphism of this rooted tree. Below we will sometimes identify a state $q$ of an automaton $\A$ with the corresponding initial automaton $\A_q$. Thus each state of an automaton defines an endomorphism (or automorphism in the invertible case) of the tree $\Sigma^*$.

\begin{definition}
Let $\A$ be an (invertible) automaton over an alphabet $\Sigma$. The semigroup $\mathds S(\A)$ (group $\mathds G(\A)$) generated by all states of $\A$ viewed as endomorphisms (automorphisms) of the rooted tree $\Sigma^*$ under the operation of composition is called an \emph{automaton semigroup} (\emph{automaton group}).
\end{definition}

Another popular name for automaton groups (resr. semigroups) is
self-similar groups (resr. semigroups)
(see~\cite{nekrash:self-similar}).

We will also consider subgroups of automaton groups. These groups are generated by one or more initial invertible automata.

Conversely, any endomorphism of $\Sigma^*$ can be encoded by the action
of a suitable initial automaton. In order to show this we need the notion of a
\emph{section} of a endomorphism at a vertex of the tree. Let $g$ be
a endomorphism of the tree $\Sigma^*$ and $x\in \Sigma$. Then for any $v\in
\Sigma^*$ we have
\[g(xv)=g(x)v'\]
for some $v'\in \Sigma^*$. Then the map $g|_x\colon \Sigma^*\to \Sigma^*$ given by
\[g|_x(v)=v'\]
defines a endomorphism of $\Sigma^*$ and is called the \emph{section} of
$g$ at vertex $x$. Furthermore,  for any nontrivial word $x_1x_2\ldots x_n\in \Sigma^*$
we define \[g|_{x_1x_2\ldots x_n}=(\ldots((g|_{x_1})|_{x_2})|_{x_3}\ldots)|_{x_n}.\]
Finally, for empty word $\emptyset$ corresponding to the root of the tree we define $g|_{\emptyset}=g$.

Given a endomorphism $g$ of $\Sigma^*$ we construct an initial automaton
$\A(g)$ whose action on $\Sigma^*$ coincides with that of $g$ as follows.
The set of states of $\A(g)$ is the set $\{g|_v\colon  v\in \Sigma^*\}$
of different sections of $g$ at the vertices of the tree. The
transition and output functions are defined by
\[\begin{array}{l}
\pi(g|_v,x)=g|_{vx},\\
\lambda(g|_v,x)=g|_v(x).
\end{array}\]

Thus, the semigroup of all endomorphisms of the tree $\Sigma^*$ is isomorphic to the semigroup generated by all initial automata over $\Sigma$. Respectively, the group of all automorphisms of the tree $\Sigma^*$ is isomorphic to the group generated by all initial invertible automata over $\Sigma$.

For any automaton group $G$ there is a natural embedding of $G$ into the permutational wreath product of $G$ with $\Sym(\Sigma)$
\[G\hookrightarrow G \wr_{\Sigma} \Sym(\Sigma)\]
defined by
\[G\ni g\mapsto (g_1,g_2,\ldots,g_d)\lambda(g)\in G\wr_{\Sigma} \Sym(\Sigma),\]
where $g_1,g_2,\ldots,g_d$ are the sections of $g$ at the vertices
of the first level, and $\lambda(g)$ is the permutation of $\Sigma$ induced by the action of $g$ on the first level of the tree.

The above embedding is convenient in computations involving the
sections of automorphisms, as well as for defining automaton groups.
When $G$ is a finitely generated automaton group, the restriction of the above embedding to a (finite) generating set of $G$ is sometimes called the \emph{wreath recursion} defining the group. For example, the wreath recursion of the Lamplighter group generated by the automaton shown in Figure~\ref{fig:lamplighter_aut} is given in~\eqref{eq:lamplighter_wr}.

\subsection{Some simple cases}
\label{ssec:triv}

We start with two easy examples.

\begin{proposition}
Let $G$ be a finite group acting on a rooted spherically homogeneous tree $T$ with infinite boundary. The space $\mathcal{M}^e_G$ of ergodic invariant probability measures on $\partial T$ is homeomorphic to the Cantor set.
\end{proposition}

\begin{proof}
Since $G$ is finite, the size of each orbit is bounded by $|G|$. Therefore, there is only a finite number of times when the size of the orbit grows while passing from a vertex to its child in the orbit tree. Thus, there is an integer $N$ such that for each vertex of $T_G$ of level at least $N$ corresponding to an orbit $\mathcal O$, and any child of this vertex corresponding to an orbit $\mathcal O'$, we have $|\mathcal O|=|\mathcal O'|$. Hence, the structure of the subtrees of the orbit tree $T_{G}$ hanging down from the vertices of level $N\neq 0$ will coincide with the structure of corresponding subtrees of $T$. Equivalently, for each $n\geq N$ the degrees of all vertices of level $n$ in $T_{G}$ coincide with the degree of vertices of level $n$ in $T$. Since $T$ has infinite boundary, it has an infinite number of levels where branching is happening. This implies that the same is true for $T_G$ as well, so $\partial T_G$ is homeomorphic to the Cantor set.
\end{proof}

In the opposite case when a group $G$ acts spherically transitively on an infinite spherically homogeneous rooted tree, the orbit tree is just a 1-ary rooted tree in which every vertex has exactly one child, its boundary consists of one point that corresponds to the unique (ergodic) invariant probability measure on $\partial T$. This particular case is considered in~\cite[Proposition~6.5]{gns00:automata}. Note, that in the case of a regular rooted tree $T_d$, according to~\cite{gawron_ns:conjugation}, an automorphism acts spherically transitively on $T_d$ if and only if it is conjugate to the, so-called, adding machine. A more general approach to adding machines acting on Cantor sets is developed in~\cite{buescu_s:liapunov_stability95,buescu_s:liapunov_stability06}, where their classification is given in terms of their spectral properties.

In the case when a spherically homogeneous tree $T$ is constructive, i.e. the sequence $\{d_n\}_{n\geq 0}$, where $d_n$ is the degree of vertices on the $n$-th level, is recursive, we formulate the following algorithmic questions.

\begin{question}
Let $G$ be an automaton group (or, more generally, a group acting on a constructive spherically homogeneous tree $T$ with infinite boundary).
\begin{itemize}
\item Is there a way to algorithmically describe the structure of the orbit tree $T_G$?
\item Is there an algorithm that checks whether $\partial T_G$ is finite, or even consists of one point (equivalently, whether $G$ acts level-transitively on the tree), or is homeomorphic to the Cantor set?
\end{itemize}
\end{question}

\subsection{The Universal Grigorchuk group}
\label{ssec:univ_grigorch}
Another example that we study here is the universal Grigorchuk group\footnote{The second author insists on the use of this terminology}. This group is defined as a universal group for the family of Grigorchuk groups $G_{\omega}$ constructed in~\cite{grigorch:degrees}. Namely, it is the quotient $F_4/N$ of the free group $F_4$ of rank 4 by the normal subgroup $N=\cap_{\omega\in\{0,1\}^\omega}N_{\omega}$, where $G_{\omega}=F_4/N_{\omega}$. For detailed information about this group we refer the reader to~\cite{grigorch:solved,benli_gn:universal}. The main open question related to this group is whether it is amenable. For our purposes we shall only need the realization of this group as an automaton group.

\begin{proposition}[\cite{benli_gn:universal}]
The universal Grigorchuk group $U$ can be defined as a group generated by the 4-state automaton over the 6-letter alphabet $\Sigma=\{1,2,3,4,5,6\}$ given by the following wreath recursion:
\[\begin{array}{lcl}
a &=& (1, 1, 1, 1, 1, 1)(14)(25)(36),\\
b &=& (a, a, 1, b, b, b),\\
c &=& (a, 1, a, c, c, c),\\
d &=& (1, a, a, d, d, d).
\end{array}
\]
\end{proposition}

Note that $\Sigma$ is partitioned into three disjoint alphabets $\Sigma_1=\{1,4\}$, $\Sigma_2=\{2,5\}$ and $\Sigma_3=\{3,6\}$. It follows immediately from the wreath recursion that if for $g\in U$ and $x_1x_2\ldots x_n\in\Sigma^n$ we have $g(x_1x_2\ldots x_n)=y_1y_2\ldots y_n$ for some $y_1y_2\ldots y_n\in\Sigma^n$, then for each $i$ letters $x_j$ and $y_j$ must belong to the same $\Sigma_i$. The next proposition shows that this is the only obstruction to transitivity.

\begin{proposition}
The orbits of the action of $U$ on level $n$ of $\Sigma^*$ are Cartesian products of the alphabets $\Sigma_{i_1}\Sigma_{i_2}\ldots\Sigma_{i_n}$, where $i_j\in\{1,2,3\}$.
\end{proposition}

\begin{proof}
The proof follows by induction on levels and uses the fact that $U$ is \emph{self-replicating}. In other words, for each $v\in\Sigma^*$ the natural endomorphism $\phi_v$ from the stabilizer $\St_U(v)=\{g\in U\mid g(v)=v\}$ of $v$ in $U$ to $U$, given by $\phi_v(g) = g|_v$, is surjective. In particular, for each $v\in\Sigma^*$ there is $g\in U$ such that $g(v)=v$ and $g|_v=a$. The existence of such element proves the induction step as $a$ permutes the letters in each $\Sigma_i$.
\end{proof}

We directly obtain the following corollary related to the ergodic decomposition.

\begin{corollary}
The orbit tree $T_U$ of the action of $U$ is isomorphic to the 3-ary regular rooted tree, and therefore the space $\mathcal{M}^e_U$ of ergodic invariant probability measures is homeomorphic to the Cantor set.
\end{corollary}

\subsection{Sushchansky groups}
\label{ssec:sushch}
Suchchansky introduced a class of infinite $p$-groups generated by pairs of initial automata acting on the $p$-ary rooted tree (using the language of Kaloujnine tableaux~\cite{kalou:la_structure}) in~\cite{sushch:burnside}. These groups were later studied in~\cite{bondarenko_s:sushch}, where, in particular, it was proved that they have intermediate growth and the structure of the orbit trees was computed.

Let $\sigma=(0,1,\ldots,p-1)$ be a cyclic
permutation of the alphabet $\Sigma=\{0,1,\ldots,p-1\}$. With a slight abuse of notation, depending on
the context, $\sigma$ will also denote the automorphism of $\Sigma^*$ of
the form $(1,1,\ldots,1)\sigma$.

Given an arbitrary linear order $\lambda=\{(\alpha_i,\beta_i)\}$ on $\Sigma^2$ we define the Sushchansky group $G_{\lambda}$ generated by the two automorphisms $A$ and $B_{\lambda}$ of $T_p$ with the set of vertices $\Sigma^*$. We first define words
$u,v\in \Sigma^{p^2}$ in the following way:
\[
u_i=\left\{%
\begin{array}{ll}
    0, & \hbox{ if } \beta_i=0; \\
    1, & \hbox{ if } \beta_i\neq 0. \\
\end{array}%
\right. \qquad \qquad
v_i=\left\{%
\begin{array}{ll}
    1, & \hbox{ if } \beta_i=0; \\
    -\frac{\alpha_i}{\beta_i}, & \hbox{ if } \beta_i\neq 0. \\
\end{array}%
\right.
\]
The words $u$ and $v$ encode the actions of $B_{\lambda}$ on the
words $00\ldots 01*$ and $10\ldots 01*$, respectively. Using the
words $u$ and $v$ we can construct automorphisms
$q_1,\ldots,q_{p^2}, r_1,\ldots, r_{p^2}$ of the tree $\Sigma^{*}$ by the
following recurrent formulas:
\begin{equation}\label{eqn def aut q r}
q_i=(q_{i+1},\sigma^{u_i},1,\ldots,1), \qquad
r_i=(r_{i+1},\sigma^{v_i},1,\ldots ,1),
\end{equation}
for $i=1,\ldots,p^2$, where the indices are considered modulo $p^2$,
i.e. $i=i+np^2$ for any $n$.

These automorphisms $q_i$ and $r_i$ are
precisely the restrictions of $B_{\lambda}$ on the words
$00(0)^{i-1+np^2}$ and $10(0)^{i-1+np^2}$, respectively, for any
$n\geq 0$.

The action of the tableau $A$ is given by:
\[
A=(1,\sigma,\sigma^2,\ldots,\sigma^{p-1})\sigma;
\]
while $B_{\lambda}$ acts trivially on the second level and the
action on the rest is given by the sections:
\[
B_{\lambda}|_{00}=q_1, \quad B_{\lambda}|_{10}=r_1, \quad
B_{\lambda}|_{21}=\sigma
\]
and all the other sections are trivial. In particular, the
automorphisms $A$ and $B_{\lambda}$ are finite-state and Sushchansky
group $G_{\lambda}$ is generated by two finite initial automata, whose structure
is shown in Figure~\ref{aut_general} (where the double circled nodes correspond to generators $A$ and $B_{\lambda}$).

\begin{figure}[h]
\begin{center}
\epsfig{file=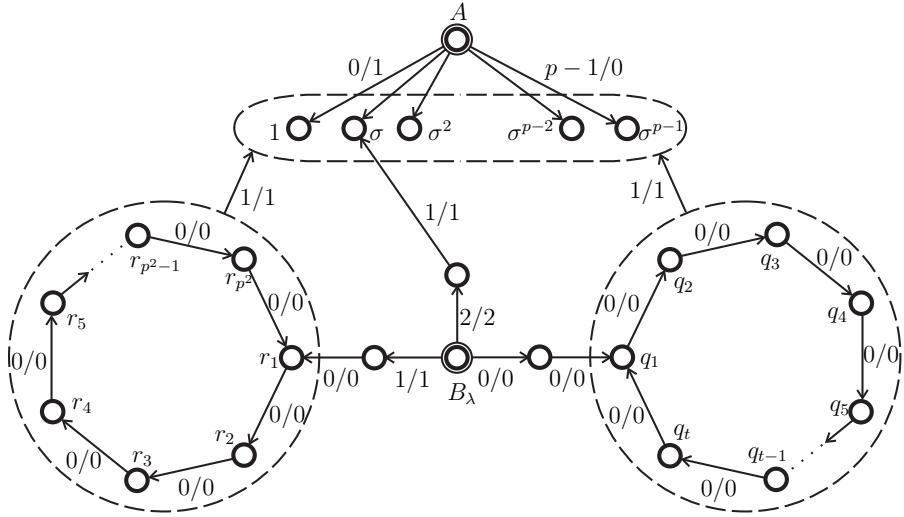,width=340pt}
\end{center}
\caption{The Structure of Sushchansky automaton\label{aut_general}}
\end{figure}

The following proposition describes the orbit tree
\begin{proposition}[\cite{bondarenko_s:sushch}]\label{prop orbit tree}
The structure of the orbit tree $T_{G_\lambda}$ does not depend on
the type $\lambda$ and is shown in Figure~\ref{fig orbit tree}. Namely, there is only one vertex on the first level of the tree that has $p$ children, one of which is the root of a line, and the others are the roots of regular $p$-ary trees.
\begin{figure}[h]
\begin{center}
\includegraphics{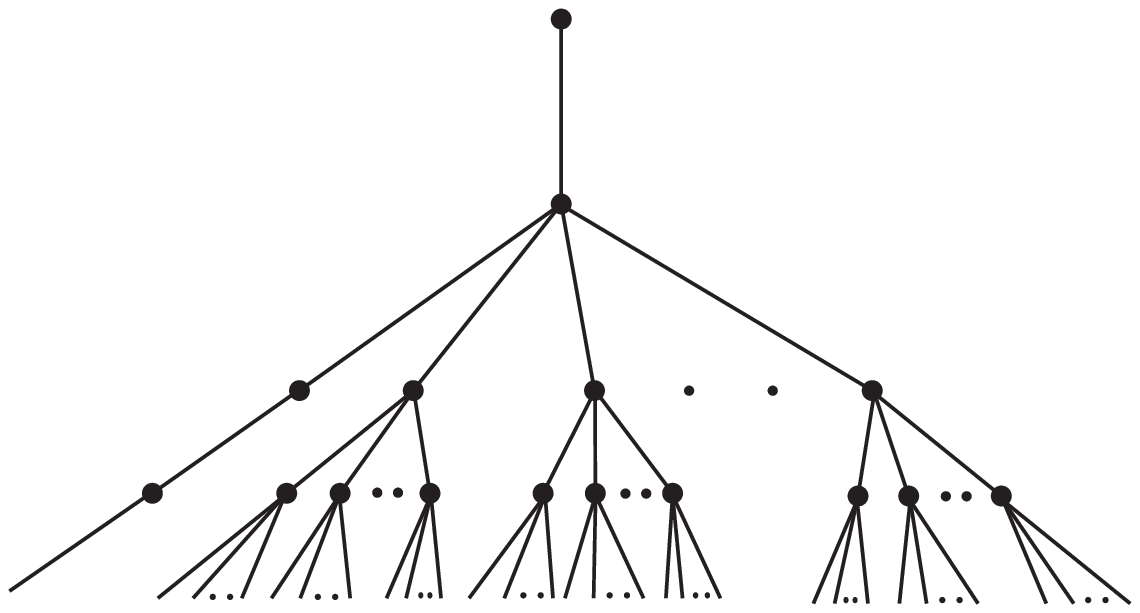}
\caption{The Orbit tree $T_{G_\lambda}$ of Sushchansky group\label{fig orbit tree}}
\end{center}
\end{figure}
\end{proposition}

\subsection{Lamplighter group}
\label{ssec:lamp}
Recall that the lamplighter group $\LL$, the permutational wreath product $(\Z/2\Z)\wr\Z\cong\bigl(\oplus_{\Z}(\Z/2Z)\bigr)\rtimes\Z$, can be realized as an automaton group generated by the automaton shown in Figure~\ref{fig:lamplighter_aut} with the following wreath recursion:
\begin{equation}
\label{eq:lamplighter_wr}
\begin{array}{lll}
a&=&(b,a)\sigma,\\
b&=&(b,a).\\
\end{array}
\end{equation}

\begin{figure}[h]
\begin{center}
\includegraphics[width=150pt]{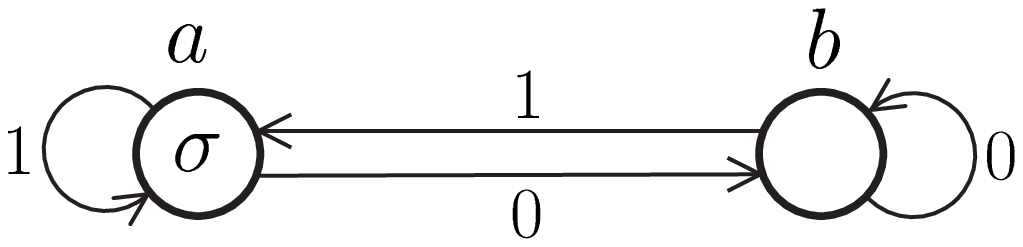}
\caption{Automaton generating the lamplighter group $\LL$\label{fig:lamplighter_aut}}
\end{center}
\end{figure}

In this subsection we give a complete description of the ergodic decompositions for the cyclic subgroups $\langle a\rangle$ and $\langle b\rangle$ of $\LL$. Let $T_{\langle a\rangle}$ and $T_{\langle b\rangle}$ be the corresponding orbit trees of the actions of these subgroups on the binary tree $T$.

\begin{theorem}
\label{thm:lamplighter}
\begin{itemize}
\item[(a)] In the orbit tree $T_{\langle a\rangle}$  all vertices on levels $2^n-1, n\geq0$ have one child and all vertices on other levels have two children (see Figure~\ref{fig:orbit_tree_a}). The space of ergodic components of the action of $\langle a\rangle$ on $X^\omega$ is homeomorphic to the Cantor set.
\item[(b)] The orbit tree $T_{\langle b\rangle}$ is recursively obtained by declaring that the root of the tree has two children that are roots of trees $T_{\langle b\rangle}$ and $T_{\langle a\rangle}$ (see Figure~\ref{fig:orbit_tree_b}). The space of ergodic components of the action of $\langle b\rangle$ on $X^\omega$ is again homeomorphic to the Cantor set.
\end{itemize}
\end{theorem}

A useful observation about the lamplighter group made in~\cite{gns00:automata} is that the action of generators $a$ and $b$ can be defined in terms of functions acting on formal power series. We are going to describe this action. The boundary $\partial T$, consisting of infinite sequences over $X$, can be identified with the ring of formal power series $(\Z/2\Z)[[t]]$ via the map
\[a_0a_1a_2a_3\ldots\mapsto a_0+a_1t+a_2t^2+a_3t^3+\cdots\]

We will use this identification. Also, we will associate finite sequences over $X$ with corresponding polynomials in $(\Z/2\Z)[t]$ , which can be viewed as power series with finite number of nonzero terms. For example, $10^k$ and $11$ in $X^*$ will correspond to $1$ and $1+t$ in $(\Z/2\Z)[t]$, respectively. As was observed in~\cite{gns00:automata}, under this identification the actions of $a$ and $b$ on $f(t)\in(\Z/2\Z)[[t]]$ are defined as
\[\begin{array}{lll}
(a(f))(t)&=&(1+t)f(t)+1,\\
(b(f))(t)&=&(1+t)f(t).\\
\end{array}
\]

It will be convenient in the proof to operate with orbits of group actions using the following notion.

\begin{definition}
For an automorphism $g\in\Aut(X^*)$ and for $v\in X^*\cup X^\omega$ whose orbit under the action of $\langle g\rangle$ has size $m\in\mathbb N\cup\{\infty\}$ the \emph{orbit matrix} of $v$ with respect to $g$ is the $n\times |v|$ matrix $M(v,g)$ whose $ij$-th entry contains the $j$-th symbol of $g^{i-1}(v)$ (so that the first row corresponds to $v$ itself).
\end{definition}

\begin{lemma}
\label{lem:orb10}
The size of the orbit $\Orb_b(10^k)$ of the vertex $10^k$ under the action of $b$ is $\displaystyle{2^{[\log_2k]+1}}$.
\end{lemma}

\begin{proof}
The orbit $\Orb_b(10^k)$ corresponds to the orbit of $1$ under multiplication by $(1+t)$ in $(\Z/2\Z)[[t]]/(t^{k+1})$. This orbit will consist of polynomials
\[(1+t)^n=\Sigma_{i=0}^n \overline{n\choose i}t^i \mod t^{k+1},\]
where by $\overline{x}$ we denote $x \mod 2$. It is well-known (see, for example,~\cite{fine:binomial47}) that the coefficients of these polynomials, plotted as a rectangular array in which the $i$-th row contains the values $\overline{n\choose i}$,  have a fractal shape such as that of a Sierpinski triangle as shown in Figure~\ref{fig:sierpinski}.

\begin{figure}
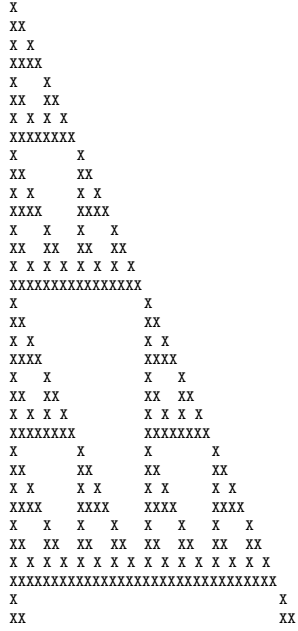

\begin{center}
\tiny{
\begin{verbatim}
                                                                X
                                                                XX
                                                                X X
                                                                XXXX
                                                                X   X
                                                                XX  XX
                                                                X X X X
                                                                XXXXXXXX
                                                                X       X
                                                                XX      XX
                                                                X X     X X
                                                                XXXX    XXXX
                                                                X   X   X   X
                                                                XX  XX  XX  XX
                                                                X X X X X X X X
                                                                XXXXXXXXXXXXXXXX
                                                                X               X
                                                                XX              XX
                                                                X X             X X
                                                                XXXX            XXXX
                                                                X   X           X   X
                                                                XX  XX          XX  XX
                                                                X X X X         X X X X
                                                                XXXXXXXX        XXXXXXXX
                                                                X       X       X       X
                                                                XX      XX      XX      XX
                                                                X X     X X     X X     X X
                                                                XXXX    XXXX    XXXX    XXXX
                                                                X   X   X   X   X   X   X   X
                                                                XX  XX  XX  XX  XX  XX  XX  XX
                                                                X X X X X X X X X X X X X X X X
                                                                XXXXXXXXXXXXXXXXXXXXXXXXXXXXXXXX
                                                                X                               X
                                                                XX                              XX
\end{verbatim}
}
\caption{Initial part of the orbit of $10^\infty$ under the action of powers of $b$, where 1's are replaced with ``X'' and 0's by empty spaces.\label{fig:sierpinski}}
\end{center}
\end{figure}

For the purpose of completeness and to explain the further steps we will include the proof of this fact here. An important observation behind the structure of the Sierpinski triangle is that the orbit matrix $M(10^{2^{n+1}-1},b)$ is a square $2^{n+1}\times2^{n+1}$ matrix that has the following block decomposition:

\def\arraystretch{1.5}
\begin{equation}
\label{eq:decomposition}
M(10^{2^{n+1}-1},b)=\left[\begin{array}{c|c}
M(10^{2^{n}-1},b)&0\\
\hline
M(10^{2^{n}-1},b)&M(10^{2^{n}-1},b)
\end{array}
\right].
\end{equation}
\def\arraystretch{1}
We prove the above decomposition by induction on $n$. The base of induction is satisfied since $M(1,b)=[1]$ and
\[M(10,b)=\left[\begin{array}{c|c}
1&0\\
\hline
1&1
\end{array}
\right]\]
The induction step is proved as follows. Assume that $|\Orb_b(10^{2^{n}-1})|=2^n$, and hence $M(10^{2^{n}-1},b)$ is a square $2^n\times2^n$ matrix.

First of all, since when $i\leq 2^n$ the expansion of $(1+t)^i$ does not have terms of degree greater than $2^n$, we immediately conclude that the upper right corner of the matrix~\eqref{eq:decomposition} is a $2^n\times2^n$ zero matrix. And by definition of the orbit matrix we will see exactly $M(10^{2^{n}-1},b)$ in the left top corner.

Further, since $(1+t)^{2^n}=1+t^{2^n}$ in $(\Z/2\Z)[[t]]$, we have
\[(1+t)^{2^n+i}=(1+t^{2^n})(1+t)^i=(1+t)^i+t^{2^n}(1+t)^i.\]
When $0\leq i< 2^{n}$ the term $(1+t)^i$ will reproduce the orbit of $10^{2^{n}-1}$ in the bottom left corner of the orbit matrix in~\eqref{eq:decomposition}, while the term $t^{2^n}(1+t)^i$ will reproduce the same orbit shifted to the right by $2^n$ positions, thus filling the bottom right corner of $M(10^{2^{n+1}-1},b)$. Furthermore, $b^{2^{n+1}}(10^{2^{n+1}-1})=10^{2^{n+1}-1}$ since
\[(1+t)^{2^{n+1}}=1+t^{2^{n+1}}\equiv 1\mod t^{2^{n+1}}.\]
Therefore, by the induction assumption, the size of $\Orb_b(10^{2^{n+1}-1})$ equals to $2^{n+1}$, which implies that $M(10^{2^{n+1}-1},b)$ is a square $2^{n+1}\times2^{n+1}$ matrix. In particular, this size agrees with the statement of the lemma.

Finally, for $2^{n}-1<k\leq 2^{n+1}-1$ the orbit $\Orb_b(10^k)$ has the same size as the orbit $\Orb_b(10^{2^{n+1}-1})$ since by the decomposition~\eqref{eq:decomposition} the only line beginning with $10^k$ in the orbit matrix $M(10^{2^{n+1}-1},b)$ is the first one.
\end{proof}

\begin{lemma}
\label{lem:orb1w}
The size of the orbit $\Orb_b(0^i1w)$ of the vertex $0^i1w$ under the action of $b$ is $\displaystyle{2^{[\log_2|w|]+1}}$.
\end{lemma}

\begin{proof}
First of all, since $b^n=(b^n,a^n)$, we have $b^n(0^i1w)=0^ib^n(1w)$. Therefore
\[|\Orb_b(0^i1w)|=|\Orb_b(1w)|\]
and we can assume that $i=0$.

The vertex $1w=1a_1a_2\ldots a_{|w|}$ corresponds to the power series (which is, in fact, a polynomial) $f(t)=1+tg(t)=1+a_1t+a_2t^2+\cdots+a_{|w|}t^{|w|}\in(\Z/2\Z)[[t]]$ for some polynomial $g(t)$. Therefore, the series corresponding to $b^n(1w)$ has the form $(1+t)^n(1+tg(t))$. The size of the orbit of $1w$ then is equal to $N-1$, where $N>1$ is the smallest number such that
\begin{equation}
\label{eq:series}
(1+t)^N(1+tg(t))\equiv(1+tg(t)) \mod t^{|w|+2}
\end{equation}
which is equivalent to
\[(1+tg(t))(1+(1+t)^N)\equiv 0\mod t^{|w|+2}.\]
The last equality holds true if and only if $(1+(1+t)^N)\equiv 0\mod t^{|w|+2}$ as otherwise the smallest degree non-zero term in $(1+(1+t)^N)\mod t^{|w|+2}$ would produce a non-zero term in the lefthand side of~\eqref{eq:series}. Therefore, the smallest $N$ satisfying~\eqref{eq:series} is equal to the smallest $N$ for which
\[(1+t)^N\equiv 1\mod t^{|w|+2},\]
which, by the above argument, equals to the size of the orbit $\Orb_b(10^{|w|})$. Application of Lemma~\ref{lem:orb10} finishes the proof.
\end{proof}

\begin{corollary}
\label{cor:orb_a}
The size of the orbit $\Orb_a(w)$ of the vertex $w$ under the action of $a$ is $\displaystyle{2^{[\log_2|w|]+1}}$.
\end{corollary}

\begin{proof}
Follows immediately from the identity $b(1w)=1a(w)$ and Lemma~\ref{lem:orb1w}.
\end{proof}

\begin{remark} It follows immediately from the equality $(b(f))(t)=(1+t)f(t)$ that
\[b(a_0a_1a_2\ldots)=a_0(a_1+a_0)(a_2+a_1)(a_3+a_2)\ldots,\]
where the addition is performed $\mod 2$. This implies that the orbit of $w=a_0a_1a_2\ldots$ under the action of $b$, viewed as an infinite matrix with $ij$-th entry containing the $j$-th symbol of $b^i(w)$, can be obtained as the sum of the corresponding matrices for the orbits of the vertices of the form $0^{l-1}10^\infty$, where the sum is taken over all $l$ for which $a_l=1$. In other words, we sum up together $\mod 2$ Sierpinski triangles that grow from positions in which $a_l=1$ (see Figure~\ref{fig:sierpinski2}).
\end{remark}

\begin{figure}
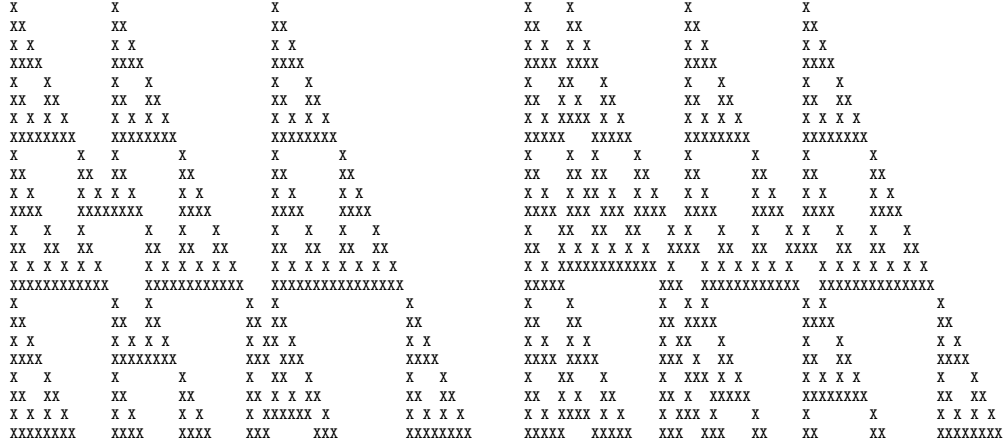

\begin{center}
\tiny{
\begin{verbatim}
                X           X                  X                             X    X             X             X
                XX          XX                 XX                            XX   XX            XX            XX
                X X         X X                X X                           X X  X X           X X           X X
                XXXX        XXXX               XXXX                          XXXX XXXX          XXXX          XXXX
                X   X       X   X              X   X                         X   XX   X         X   X         X   X
                XX  XX      XX  XX             XX  XX                        XX  X X  XX        XX  XX        XX  XX
                X X X X     X X X X            X X X X                       X X XXXX X X       X X X X       X X X X
                XXXXXXXX    XXXXXXXX           XXXXXXXX                      XXXXX   XXXXX      XXXXXXXX      XXXXXXXX
                X       X   X       X          X       X                     X    X  X    X     X       X     X       X
                XX      XX  XX      XX         XX      XX                    XX   XX XX   XX    XX      XX    XX      XX
                X X     X X X X     X X        X X     X X                   X X  X XX X  X X   X X     X X   X X     X X
                XXXX    XXXXXXXX    XXXX       XXXX    XXXX                  XXXX XXX XXX XXXX  XXXX    XXXX  XXXX    XXXX
                X   X   X       X   X   X      X   X   X   X                 X   XX  XX  XX   X X   X   X   X X   X   X   X
                XX  XX  XX      XX  XX  XX     XX  XX  XX  XX                XX  X X X X X X  XXXX  XX  XX  XXXX  XX  XX  XX
                X X X X X X     X X X X X X    X X X X X X X X               X X XXXXXXXXXXXX X   X X X X X X   X X X X X X X
                XXXXXXXXXXXX    XXXXXXXXXXXX   XXXXXXXXXXXXXXXX              XXXXX           XXX  XXXXXXXXXXXX  XXXXXXXXXXXXXX
                X           X   X           X  X               X             X    X          X  X X           X X             X
                XX          XX  XX          XX XX              XX            XX   XX         XX XXXX          XXXX            XX
                X X         X X X X         X XX X             X X           X X  X X        X XX   X         X   X           X X
                XXXX        XXXXXXXX        XXX XXX            XXXX          XXXX XXXX       XXX X  XX        XX  XX          XXXX
                X   X       X       X       X  XX  X           X   X         X   XX   X      X  XXX X X       X X X X         X   X
                XX  XX      XX      XX      XX X X XX          XX  XX        XX  X X  XX     XX X  XXXXX      XXXXXXXX        XX  XX
                X X X X     X X     X X     X XXXXXX X         X X X X       X X XXXX X X    X XXX X    X     X       X       X X X X
                XXXXXXXX    XXXX    XXXX    XXX     XXX        XXXXXXXX      XXXXX   XXXXX   XXX  XXX   XX    XX      XX      XXXXXXXX
\end{verbatim}
}
\caption{Initial part of the orbit of a random vertex under the action of powers of $b$.\label{fig:sierpinski2}}
\end{center}
\end{figure}

\begin{proof}[Proof of Theorem~\ref{thm:lamplighter}]
Item (a) immediately follows from Corollary~\ref{cor:orb_a} and item (b) is an obvious consequence of the wreath recursion decomposition $b=(b,a)$.
\end{proof}

\begin{figure}[h]
\begin{center}
\epsfig{file=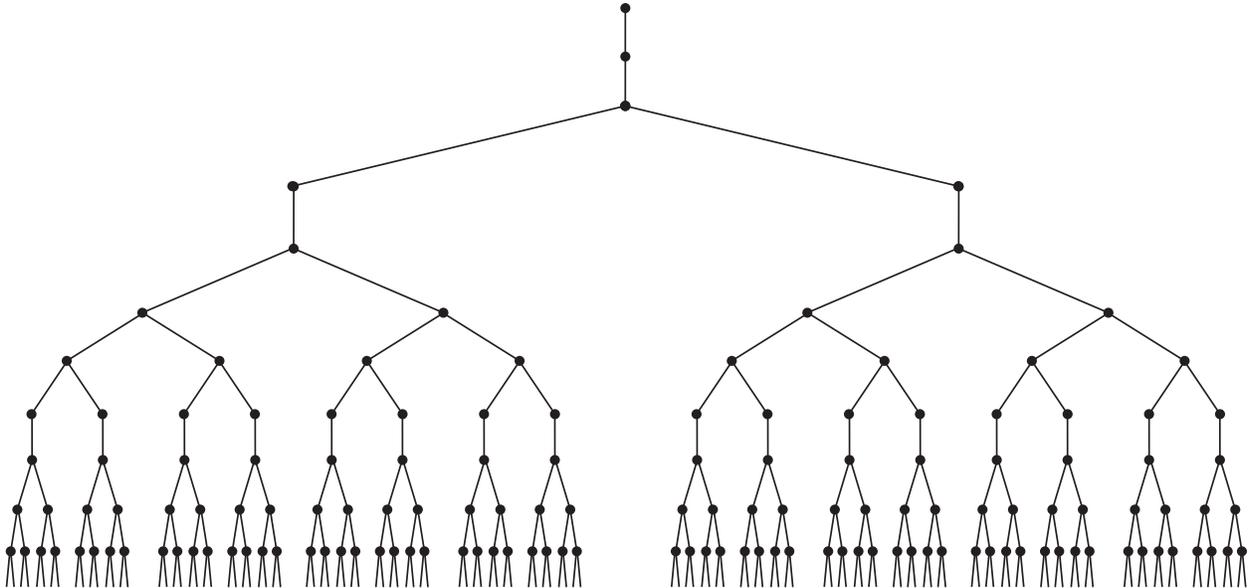,width=\textwidth}
\caption{Orbit tree $T_{\langle a\rangle}$ of the generator $a$ of the lamplighter group.\label{fig:orbit_tree_a}}
\end{center}
\end{figure}

\begin{figure}[h]
\begin{center}
\epsfig{file=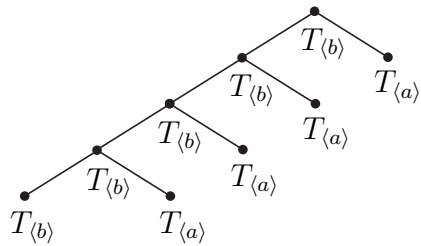}
\caption{Orbit tree $T_{\langle b\rangle}$ of the generator $b$ of the lamplighter group.\label{fig:orbit_tree_b}}
\end{center}
\end{figure}

In the end of the paper we would like to bring the attention to the fractal nature of orbit matrices for elements of automaton groups observed in Figure~\ref{fig:sierpinski} that has not been studied before. We conclude the paper with the following example.

\begin{example}
\label{ex:ll2}
Consider a group $G$ generated by a 4-state automaton with the following wreath recursion:
\[\begin{array}{lcl}
a&=&(d,d)\sigma,\\
b&=&(c,c),\\
c&=&(a,b),\\
d&=&(b,a).
\end{array}\]
This group has been studied in~\cite{klimann_ps:orbit_automata} where it was proved, in particular, that the element $ac$ has infinite order. Very recently it was shown by Sidki and the second author that the whole group is isomorphic to the extension of index 2 of a rank 2 lamplighter group $(\Z/2\Z)^2\wr\Z$. A part of the orbit matrix of $0^\infty$ with respect to the element $ac$ is shown in Figure~\ref{fig:lamp_orbit} and also clearly has a self-similar pattern.
\end{example}

\begin{figure}
\begin{center}
\tiny{
\begin{verbatim}
               X  XXX  X  XXX  X  XXX  X  XXX  X  XXX  X  XXX  X  XXX  X  XXX  X  XXX  X  XXX  X  XXX  X  XXX  X  XXX  X  XXX  X  X
                X XXX X    X    X XXX X    X    X XXX X    X    X XXX X    X    X XXX X    X    X XXX X    X    X XXX X    X    X X
               XXXX  X XX   X  XXXX  X XX   X  XXXX  X XX   X  XXXX  X XX   X  XXXX  X XX   X  XXXX  X XX   X  XXXX  X XX   X  XXXX
                  XX X   XX X X X  XXXX XX        XX X   XX X X X  XXXX XX        XX X   XX X X X  XXXX XX        XX X   XX X X X
               X   XX   X XXXXXXXX XXXX  XXXX  X   XX   X XXXXXXXX XXXX  XXXX  X   XX   X XXXXXXXX XXXX  XXXX  X   XX   X XXXXXXXX
                X    XX X  X X X XX X  XXX X    X    XX X  X X X XX X  XXX X    X    XX X  X X X XX X  XXX X    X    XX X  X X X XX
               XXX  X XXXX X  XXX XXXX XXX  X  XXX  X XXXX X  XXX XXXX XXX  X  XXX  X XXXX X  XXX XXXX XXX  X  XXX  X XXXX X  XXX X
                      XXXX  X XXX  X XX X X X X X X X  X XX    X  XXXX                XXXX  X XXX  X XX X X X X X X X  X XX    X  X
               X  XXX XXXX XXXX  XXXX XXXXXXXXXXXXXXXX X   XX   X XXXX X  XXX  X  XXX XXXX XXXX  XXXX XXXXXXXXXXXXXXXX X   XX   X X
                X XXX  X XX X  XXX X  XXXXXXXXXXXXXXXX  X    XX X  X XX    X    X XXX  X XX X  XXX X  XXXXXXXXXXXXXXXX  X    XX X
               XXXX  XXXX XXXX XXX  X XXXXXXXXXXXXXXXX XXX  X XXXX X   XX   X  XXXX  XXXX XXXX XXX  X XXXXXXXXXXXXXXXX XXX  X XXXX
                  XX XXXX  X XX X X X  X X X X X X X XX X X X  X XX   X XX        XX XXXX  X XX X X X  X X X X X X X XX X X X  X XX
               X   XX X  XXXX XXXXXXXX X  XXX  X  XXX XXXXXXXX X   XX X  XXXX  X   XX X  XXXX XXXXXXXX X  XXX  X  XXX XXXXXXXX X
                X    X XXX X  XXXXXXXX  X XXX X    X  XXXXXXXX  X    X XXX X    X    X XXX X  XXXXXXXX  X XXX X    X  XXXXXXXX  X
               XXX  X  XXX  X XXXXXXXX XXXX  X XX   X XXXXXXXX XXX  X  XXX  X  XXX  X  XXX  X XXXXXXXX XXXX  X XX   X XXXXXXXX XXX
                              XXXXXXXX    XX X   XX X  X X X XX X X X X X X X X X X X X X X X  X X X XX X  XXXX XX    XXXXXXXX
               X  XXX  X  XXX XXXXXXXX X   XX   X XXXX X  XXX XXXXXXXXXXXXXXXXXXXXXXXXXXXXXXXX X  XXX XXXX XXXX  XXXX XXXXXXXX X  X
                X XXX X    X  XXXXXXXX  X    XX X  X XX    X  XXXXXXXXXXXXXXXXXXXXXXXXXXXXXXXX  X XXX  X XX X  XXX X  XXXXXXXX  X X
               XXXX  X XX   X XXXXXXXX XXX  X XXXX X   XX   X XXXXXXXXXXXXXXXXXXXXXXXXXXXXXXXX XXXX  XXXX XXXX XXX  X XXXXXXXX XXXX
                  XX X   XX X  X X X XX X X X  X XX   X XX    XXXXXXXXXXXXXXXXXXXXXXXXXXXXXXXX    XX XXXX  X XX X X X  X X X XX X
               X   XX   X XXXX X  XXX XXXXXXXX X   XX X  XXXX XXXXXXXXXXXXXXXXXXXXXXXXXXXXXXXX X   XX X  XXXX XXXXXXXX X  XXX XXXX
                X    XX X  X XX    X  XXXXXXXX  X    X XXX X  XXXXXXXXXXXXXXXXXXXXXXXXXXXXXXXX  X    X XXX X  XXXXXXXX  X XXX  X XX
               XXX  X XXXX X   XX   X XXXXXXXX XXX  X  XXX  X XXXXXXXXXXXXXXXXXXXXXXXXXXXXXXXX XXX  X  XXX  X XXXXXXXX XXXX  XXXX X
                      XXXX  X    XX X  X X X XX X X X X X X X  X X X X X X X X X X X X X X X XX X X X X X X X  X X X XX X  XXX X  X
               X  XXX XXXX XXX  X XXXX X  XXX XXXXXXXXXXXXXXXX X  XXX  X  XXX  X  XXX  X  XXX XXXXXXXXXXXXXXXX X  XXX XXXX XXX  X X
                X XXX  X XX X X X  X XX    X  XXXXXXXXXXXXXXXX  X XXX X    X    X XXX X    X  XXXXXXXXXXXXXXXX  X XXX  X XX X X X
               XXXX  XXXX XXXXXXXX X   XX   X XXXXXXXXXXXXXXXX XXXX  X XX   X  XXXX  X XX   X XXXXXXXXXXXXXXXX XXXX  XXXX XXXXXXXX
                  XX XXXX  X X X XX   X XX    XXXXXXXXXXXXXXXX    XX X   XX X X X  XXXX XX    XXXXXXXXXXXXXXXX    XX XXXX  X X X XX
               X   XX X  XXXX  X   XX X  XXXX XXXXXXXXXXXXXXXX X   XX   X XXXXXXXX XXXX  XXXX XXXXXXXXXXXXXXXX X   XX X  XXXX  X
                X    X XXX X    X    X XXX X  XXXXXXXXXXXXXXXX  X    XX X  X X X XX X  XXX X  XXXXXXXXXXXXXXXX  X    X XXX X    X
               XXX  X  XXX  X  XXX  X  XXX  X XXXXXXXXXXXXXXXX XXX  X XXXX X  XXX XXXX XXX  X XXXXXXXXXXXXXXXX XXX  X  XXX  X  XXX
                                              XXXXXXXXXXXXXXXX        XXXX  X XXX  X XX X X X  X X X X X X X XX X X X X X X X X X X
               X  XXX  X  XXX  X  XXX  X  XXX XXXXXXXXXXXXXXXX X  XXX XXXX XXXX  XXXX XXXXXXXX X  XXX  X  XXX XXXXXXXXXXXXXXXXXXXXX
                X XXX X    X    X XXX X    X  XXXXXXXXXXXXXXXX  X XXX  X XX X  XXX X  XXXXXXXX  X XXX X    X  XXXXXXXXXXXXXXXXXXXXX
               XXXX  X XX   X  XXXX  X XX   X XXXXXXXXXXXXXXXX XXXX  XXXX XXXX XXX  X XXXXXXXX XXXX  X XX   X XXXXXXXXXXXXXXXXXXXXX
                  XX X   XX X X X  XXXX XX    XXXXXXXXXXXXXXXX    XX XXXX  X XX X X X  X X X XX X  XXXX XX    XXXXXXXXXXXXXXXXXXXXX
               X   XX   X XXXXXXXX XXXX  XXXX XXXXXXXXXXXXXXXX X   XX X  XXXX XXXXXXXX X  XXX XXXX XXXX  XXXX XXXXXXXXXXXXXXXXXXXXX
                X    XX X  X X X XX X  XXX X  XXXXXXXXXXXXXXXX  X    X XXX X  XXXXXXXX  X XXX  X XX X  XXX X  XXXXXXXXXXXXXXXXXXXXX
               XXX  X XXXX X  XXX XXXX XXX  X XXXXXXXXXXXXXXXX XXX  X  XXX  X XXXXXXXX XXXX  XXXX XXXX XXX  X XXXXXXXXXXXXXXXXXXXXX
                      XXXX  X XXX  X XX X X X  X X X X X X X XX X X X X X X X  X X X XX X  XXX X  XXXX        XXXXXXXXXXXXXXXXXXXXX
               X  XXX XXXX XXXX  XXXX XXXXXXXX X  XXX  X  XXX XXXXXXXXXXXXXXXX X  XXX XXXX XXX  X XXXX X  XXX XXXXXXXXXXXXXXXXXXXXX
\end{verbatim}
}
\caption{Initial part of the orbit of $0^\infty$ under the action of powers of $ac\in G$ from Example~\ref{ex:ll2}, where 1's are replaced with ``X'' and 0's by empty spaces.\label{fig:lamp_orbit}}
\end{center}
\end{figure}

\clearpage

\newcommand{\etalchar}[1]{$^{#1}$}
\def\cprime{$'$} \def\cprime{$'$} \def\cprime{$'$} \def\cprime{$'$}
  \def\cprime{$'$} \def\cprime{$'$} \def\cprime{$'$} \def\cprime{$'$}
  \def\cprime{$'$} \def\cprime{$'$} \def\cprime{$'$} \def\cprime{$'$}
  \def\cprime{$'$}

\end{document}